\documentclass[11pt,a4paper]{article}

\usepackage[utf8]{inputenc}
\usepackage[T1]{fontenc}
\usepackage{lmodern}
\usepackage[english]{babel}
\usepackage[a4paper,margin=28mm,headheight=14pt]{geometry}
\usepackage{amsmath,amssymb,amsthm,mathtools}
\usepackage{microtype}
\usepackage{enumitem}
\usepackage{cite}
\usepackage[hidelinks]{hyperref}
\usepackage{bookmark}
\hypersetup{
  pdftitle={Finite-radius Jensen and relative-entropy formulations of the Balazard--Saias--Yor criterion},
  pdfauthor={Henning Wunderlich},
  pdfsubject={Finite-radius Jensen formulas and a relative-entropy reformulation of a Balazard--Saias--Yor criterion},
  pdfkeywords={Riemann hypothesis, Jensen formula, Hardy space, trivial zeros, Balazard--Saias--Yor criterion, relative entropy, data processing}
}

\newtheorem{theorem}{Theorem}[section]
\newtheorem{proposition}[theorem]{Proposition}
\newtheorem{lemma}[theorem]{Lemma}
\newtheorem{corollary}[theorem]{Corollary}
\theoremstyle{definition}

\theoremstyle{remark}
\newtheorem{remark}[theorem]{Remark}
\numberwithin{equation}{section}

\newcommand{\D}{\mathbb D}
\newcommand{\R}{\mathbb R}
\newcommand{\C}{\mathbb C}
\newcommand{\dd}{\,\mathrm d}
\newcommand{\RH}{\mathrm{RH}}
\DeclareMathOperator{\Rea}{Re}
\DeclareMathOperator{\Ima}{Im}
\DeclareMathOperator{\card}{card}
\DeclareMathOperator{\Log}{Log}
\newcommand{\mean}[1]{\frac{1}{2\pi}\int_0^{2\pi}#1\dd\theta}
\newcommand{\zetahalf}{\zeta\!\left(\frac12+it\right)}

\title{Finite-radius Jensen and relative-entropy formulations of the\protect\\
Balazard--Saias--Yor criterion}
\author{Henning Wunderlich\thanks{Offenbach am Main, Germany.
Email: \href{mailto:henningwunderlich@t-online.de}{\texttt{henningwunderlich@t-online.de}}.}}
\date{\today}

\begin{document}
\maketitle

\begin{abstract}
We give two explicit reformulations of a Balazard--Saias--Yor type criterion
for the Riemann hypothesis.  First, Jensen's formula is applied to
$u(s)=(s-1)\zeta(s)$ after mapping the unit disk onto $\Rea s<1/2$.
For the normalization whose center corresponds to $s=-1/2$, the images
$z_n=(4n-1)/(4n+3)$ of the trivial zeros yield a telescoping Jensen product.
This isolates, at finite radius, the complete divergent contribution of the
trivial zeros; the remaining defect is a nonnegative integrated count of
mapped nontrivial zeros off the critical line.  A one-parameter version gives
a gamma-quotient product and an explicit finite part.

Second, after passing through the established boundary identity, we combine
two Poisson weights to obtain an absolutely convergent logarithmic integral
of the Riemann $\Xi$-function.  From the random variable
$Z=|\Xi(T)/\Xi(0)|^2$ we define its law $\mu$ and the size-biased law
$\nu(\dd z)=z\mu(\dd z)/M$.  The data-processing inequality is lossless for
this statistic, and the Riemann hypothesis is equivalent to saturation of an
explicit upper bound for $D_{\rm KL}(\mu\|\nu)$.  We also record the associated
Donsker--Varadhan and moment forms and express the target constant through
Riemann's positive Fourier kernel.

The analytic boundary criterion and Jensen--Hardy mechanism used here are
known.  The purpose of the paper is an explicit finite-radius treatment of
the trivial zeros and an information-theoretic reformulation, not a new proof
of the Riemann hypothesis or an independent boundary-integral criterion.
\end{abstract}


\section{Introduction}

Balazard, Saias, and Yor \cite{BSY1999} proved the identity
\begin{equation}\label{eq:bsy}
 \frac{1}{2\pi}\int_{\R}
 \frac{\log|\zetahalf|}{t^2+1/4}\dd t
 =\sum_{\Rea\rho>1/2}m(\rho)
   \log\left|\frac{\rho}{1-\rho}\right|,
\end{equation}
where the sum runs over distinct nontrivial zeros of $\zeta$, with
multiplicities $m(\rho)$ and both signs of the imaginary part.  Every summand
is positive.  Consequently, the vanishing of the integral is equivalent to
the Riemann hypothesis.
Their proof already starts with Jensen's formula on circles of radius
$r<1$ and uses Hardy-space factorization in the boundary passage.
Burnol \cite{Burnol2001} gives a half-plane factorization of
$(s-1)\zeta(s)/s$; Kunik \cite{Kunik2008} records a convenient symmetric form.
Elaissaoui and Guennoun \cite{EG2015} prove a family of logarithmic integral
identities and explicitly give the specialization with weight $(1+t^2)^{-1}$.

Probabilistic viewpoints on the completed zeta function also have a substantial
literature.  Biane, Pitman, and Yor \cite{BianePitmanYor2001} relate Riemann's
theta-function representations to probability laws; Nakamura
\cite{Nakamura2015} studies normalized completed zeta functions as
characteristic functions and formulates RH-related divisibility properties.
More recent characteristic-function approaches include Nakamura and Suzuki
\cite{NakamuraSuzuki2023} and Konstantopoulos, Patie, and Sarkar
\cite{KonstantopoulosPatieSarkar2024}.  These works provide context for the
probabilistic language used below.  The relative-entropy construction in the
present paper is obtained from the logarithmic-integral identity rather than
from an independent probabilistic proof of it.

The present normalization uses the left half-plane instead of the right
half-plane.  It therefore includes the trivial zeros, whose contribution
must not be discarded.  For a particular choice of center, their finite
Jensen product telescopes exactly.  This gives a direct description of the
divergence and separates it from the bounded contribution of possible
nontrivial zeros off the critical line.

Our principal objects are a circular mean $I(r)$ and the explicitly known
trivial-zero contribution $A(r)$.  We prove
\[
 I(r)-A(r)=\int_0^r\frac{n_{\rm off}(t)}{t}\dd t,
\]
where $n_{\rm off}(t)$ counts mapped nontrivial zeros inside the disk of
radius $t$.  We then evaluate the asymptotic behavior of $A(r)$ and extend
the construction to a parameter $a\in(1/2,5/2)$.  The boundary-integral
statements are identified as known results.  The exact relation with the
original disk function in \cite{BSY1999} is given in
Section~\ref{sec:comparison}; thus no priority claim is made for the
underlying Jensen--Hardy method.

Sections~\ref{sec:xi-bridge} and \ref{sec:kl} give a relative-entropy
reformulation.  Combining the $a=1,2$ boundary identities produces an
absolutely convergent logarithmic integral of $\Xi$.  The law $\mu$ of
$Z=|\Xi(T)/\Xi(0)|^2$ is compared with its size-biased law
$\nu(\dd z)=z\mu(\dd z)/M$.  Lossless data processing reduces the
corresponding relative entropy to the scalar statistic $Z$, and
Theorem~\ref{thm:KL} identifies its gap from the target value with a positive
zero sum.  The variational and moment formulas at the end of the paper are
interpretations of this identity; they do not by themselves supply a new
estimate toward RH.

\section{Notation and Jensen's formula}\label{sec:notation}

Write $\D=\{z\in\C:|z|<1\}$.  All logarithms of positive real numbers are
natural logarithms.  In Sections~\ref{sec:notation}--\ref{sec:comparison},
$\rho=\beta+i\gamma$ denotes a nontrivial zero of $\zeta$;
the probability-measure notation used later is specified locally.
A sum over distinct zeros always includes the indicated multiplicity.
Unless restricted by $\gamma>0$, such a sum includes both signs of $\gamma$.

We use the standard facts that the trivial zeros of $\zeta$ are simple and
are exactly $-2,-4,\ldots$, that the nontrivial zeros satisfy $0<\beta<1$,
and that their multiset is invariant under conjugation and reflection in
$\Rea s=1/2$.  We also use the Riemann--von Mangoldt bound
\begin{equation}\label{eq:count-bound}
 \mathcal N(T):=\sum_{0<\Ima\rho\le T}m(\rho)=O(T\log T)
 \qquad(T\ge2).
\end{equation}
These facts may be found in \cite[Chapter~9]{Titchmarsh1986}; the functional
and reflection formulas are also recorded in \cite[\S25.4]{DLMF}.

\begin{lemma}[Jensen's formula]\label{lem:jensen}
Let $f$ be holomorphic on $\D$, with $f(0)\ne0$.  For $0<r<1$,
\begin{equation}\label{eq:jensen}
 \mean{\log|f(re^{i\theta})|}
 =\log|f(0)|+
 \sum_{\substack{f(\alpha)=0\\|\alpha|<r}}
 m(\alpha)\log\frac r{|\alpha|}.
\end{equation}
The integral is understood as a Lebesgue integral, or as an improper
integral at the finitely many zeros on the circle.  In particular, the
identity remains valid when $f$ has zeros on $|z|=r$.
\end{lemma}

\begin{proof}
The zero-free-boundary version is Jensen's formula; see, for example,
\cite{Duren1970}.  A zero $\alpha=re^{i\theta_0}$ of multiplicity $m$
produces the locally integrable singularity
$m\log|\theta-\theta_0|+O(1)$.  Factoring out all zeros on the circle
reduces the assertion to the usual version.  Each boundary factor has
circular logarithmic mean $\log r$ and value of modulus $r$ at the
origin, so its Jensen contribution is zero.
\end{proof}

One may instead exclude the countable set of moduli of all disk zeros.
Excluding only the mapped trivial-zero radii does not, without further
information, ensure a zero-free integration circle.  The integral
convention in Lemma~\ref{lem:jensen} avoids this distinction and allows
all $0<r<1$ below.

\section{The telescoping normalization}\label{sec:main}

Set
\begin{equation}\label{eq:maps}
 \begin{gathered}
 u(s)=(s-1)\zeta(s),\qquad
 S(z)=i\frac{1+z}{1-z},\qquad d(w)=\frac12+iw,\\
 \phi(z)=d(S(z))=\frac12-\frac{1+z}{1-z},\qquad
 V(z)=u(\phi(z)).
 \end{gathered}
\end{equation}
The removable value is $u(1)=1$.  Thus $u$ is entire and has exactly the
zeros of $\zeta$.  The map $\phi$ is biholomorphic from $\D$ onto
$\Rea s<1/2$, with
\begin{equation}\label{eq:inverse}
 \phi^{-1}(s)=\frac{s+1/2}{s-3/2},\qquad
 V(0)=-\frac32\zeta(-1/2)\ne0.
\end{equation}

The trivial zero $-2n$ corresponds to
\begin{equation}\label{eq:trivial-images}
 z_n=\phi^{-1}(-2n)=\frac{4n-1}{4n+3},\qquad n\ge1.
\end{equation}
These numbers are positive and increase to $1$.  Define
\begin{equation}\label{eq:definitions}
 \begin{gathered}
 n(r)=\card\{n\ge1:z_n<r\},\qquad
 T(r)=\prod_{z_n<r}\frac r{z_n},\\
 A(r)=\log|V(0)T(r)|,\qquad
 I(r)=\mean{\log|V(re^{i\theta})|}.
 \end{gathered}
\end{equation}
An empty product equals $1$.  In particular, $n(r)=0$ and $T(r)=1$
for $0<r\le3/7$.

\begin{proposition}[The trivial-zero product]\label{prop:product}
For every $0<r<1$,
\begin{equation}\label{eq:telescoping}
 T(r)=r^{n(r)}\frac{4n(r)+3}{3}
\end{equation}
and
\begin{equation}\label{eq:A-explicit}
 A(r)=\log|\zeta(-1/2)|+n(r)\log r
       +\log\frac{4n(r)+3}{2}.
\end{equation}
If $r\notin\{z_n:n\ge1\}$, then
\begin{equation}\label{eq:n-floor}
 n(r)=\left\lfloor\frac{1+3r}{4(1-r)}\right\rfloor.
\end{equation}
\end{proposition}

\begin{proof}
For $N\ge0$ the product telescopes:
\[
 \prod_{n=1}^{N}\frac1{z_n}
 =\prod_{n=1}^{N}\frac{4n+3}{4n-1}
 =\frac{4N+3}{3}.
\]
Equations~\eqref{eq:telescoping} and \eqref{eq:A-explicit} follow from
\eqref{eq:inverse}.  Finally,
$z_n<r$ is equivalent to $n<(1+3r)/(4(1-r))$, which gives
\eqref{eq:n-floor} away from equality.  At $r=z_n$ the strict-count
definition in \eqref{eq:definitions} applies.
\end{proof}

For a nontrivial zero $\rho$, put
\begin{equation}\label{eq:w-rho}
 w_\rho=\frac{\rho+1/2}{\rho-3/2}.
\end{equation}
The identity
\[
 |\rho-3/2|^2-|\rho+1/2|^2=4(1/2-\Rea\rho)
\]
shows that $|w_\rho|<1$ precisely when $\Rea\rho<1/2$.

\begin{theorem}[Finite-radius criterion]\label{thm:finite}
For every $0<r<1$,
\begin{equation}\label{eq:defect}
 D(r):=I(r)-A(r)
 =\sum_{\substack{\rho\ \text{nontrivial}\\
                  \Rea\rho<1/2,\ |w_\rho|<r}}
   m(\rho)\log\frac r{|w_\rho|}\ge0.
\end{equation}
The function $D$ is continuous and nondecreasing.  The following are
equivalent:\nopagebreak[4]
\begin{enumerate}[label=\textup{(\roman*)},itemsep=2pt,beginpenalty=10000,midpenalty=10000]
 \item the Riemann hypothesis;
 \item $I(r)=A(r)$ for every $0<r<1$;
 \item $I(r_j)=A(r_j)$ for some sequence $r_j\uparrow1$;
 \item $\lim_{r\uparrow1}(I(r)-A(r))=0$.
\end{enumerate}
\end{theorem}

\begin{proof}
The zeros of $V$ in $\D$ are the $z_n$ and the $w_\rho$ with
$\Rea\rho<1/2$.  Their multiplicities are preserved by the conformal
map.  Apply Lemma~\ref{lem:jensen} and subtract the trivial-zero
contribution.  There are finitely many terms for each $r<1$.
Each term is the continuous nondecreasing function
$m(\rho)\max\{0,\log(r/|w_\rho|)\}$, proving the asserted properties
on every compact subinterval of $(0,1)$.

Under RH the sum is empty.  Conversely, a zero with $\Rea\rho<1/2$
contributes strictly positively for every $r>|w_\rho|$.  Each of
(ii)--(iv) therefore excludes all such zeros.  Reflection in the
critical line then excludes zeros to its right as well.
\end{proof}

\begin{corollary}[Counting form]\label{cor:counting}
Let
\[
 n_{\rm off}(t)=
 \sum_{\substack{\Rea\rho<1/2\\|w_\rho|<t}}m(\rho).
\]
Then
\begin{equation}\label{eq:integrated-count}
 D(r)=\int_0^r\frac{n_{\rm off}(t)}{t}\dd t,
 \qquad rD'(r)=n_{\rm off}(r)
\end{equation}
for the derivative identity away from the moduli $|w_\rho|$.
In particular, $x\mapsto D(e^x)$ is convex on $(-\infty,0)$.
\end{corollary}

\begin{proof}
A zero at modulus $b<r$ contributes
$\int_b^r\dd t/t=\log(r/b)$.  The derivative and convexity follow,
since the counting function is nondecreasing.
\end{proof}

The zeros on the critical line map to the unit circle and never enter
these finite-radius counts.  The difference $D$ measures only the
additional zeros in the left half-plane; it is not the Jensen sum over
all zeros, which would yield an identity regardless of RH.

\section{The finite part of the divergent mean}\label{sec:finite-part}

\begin{lemma}\label{lem:delta}
The series
\begin{equation}\label{eq:delta}
 \Delta=\sum_{\Rea\rho<1/2}m(\rho)
 \log\left|\frac{\rho-3/2}{\rho+1/2}\right|
\end{equation}
converges and satisfies $0\le\Delta<\infty$.  Moreover,
$D(r)\uparrow\Delta$ as $r\uparrow1$.
\end{lemma}

\begin{proof}
Write $\rho=1/2-\delta+i\gamma$, where $0<\delta<1/2$.  Its
contribution, without the multiplicity, is
\begin{align}\label{eq:delta-bound}
 q(\rho)
 &=\frac12\log\frac{\gamma^2+(1+\delta)^2}
                         {\gamma^2+(1-\delta)^2}\notag\\
 &=\frac12\log\left(1+
     \frac{4\delta}{\gamma^2+(1-\delta)^2}\right)
 \le\frac{1}{\gamma^2+1/4}.
\end{align}
The bound \eqref{eq:count-bound} implies
$\sum_\rho m(\rho)/(1+\gamma^2)<\infty$: the dyadic block
$2^k<|\gamma|\le2^{k+1}$ contributes $O(k2^{-k})$.
Thus \eqref{eq:delta} converges.  Monotone convergence applied to
\eqref{eq:defect} gives the limit.
\end{proof}

\begin{theorem}[Explicit finite part]\label{thm:finite-part}
Define
\begin{equation}\label{eq:C}
 C=\log\bigl(2|\zeta(-1/2)|\bigr)-1
  =\log\frac{\zeta(3/2)}{2\pi}-1.
\end{equation}
Without assuming RH,
\begin{equation}\label{eq:asymptotic}
 A(r)=-\log(1-r)+C+o(1),\qquad
 I(r)=-\log(1-r)+C+\Delta+o(1).
\end{equation}
Consequently,
\begin{equation}\label{eq:finite-part-criterion}
 \RH\quad\Longleftrightarrow\quad
 \lim_{r\uparrow1}\bigl(I(r)+\log(1-r)\bigr)=C.
\end{equation}
\end{theorem}

\begin{proof}
Put $\varepsilon=1-r$.  Proposition~\ref{prop:product}, including its
strict-count convention at exceptional radii, gives
\[
 n(r)=\varepsilon^{-1}+O(1),\qquad
 n(r)\log r=-1+o(1),\qquad
 \log\frac{4n(r)+3}{2}=-\log\varepsilon+\log2+o(1).
\]
Insert these expressions into \eqref{eq:A-explicit}, and then use
Lemma~\ref{lem:delta}.  The equality of the two constants in
\eqref{eq:C} follows from
$\zeta(-1/2)=-\zeta(3/2)/(4\pi)$, a specialization of the functional
equation.  Finally, $\Delta=0$ if and only if its positive summands
are absent, which is equivalent to RH.
\end{proof}

The divergent term is therefore entirely supplied by the trivial
zeros.  Its cancellation is unconditional.  The exact finite constant,
not merely the existence of a finite limit, distinguishes RH.
The same limits hold along any sequence of radii tending to $1$,
including sequences avoiding any prescribed countable collection of
zero radii.

\section{A one-parameter family}\label{sec:parameter}

Fix
\begin{equation}\label{eq:a-range}
 \frac12<a<\frac52
\end{equation}
and define
\begin{equation}\label{eq:parameter-maps}
 \phi_a(z)=\frac12-a\frac{1+z}{1-z},\qquad
 \psi_a(z)=1-\phi_a(z),\qquad V_a=u\circ\phi_a.
\end{equation}
The restriction \eqref{eq:a-range} places $\phi_a(0)$ in $(-2,0)$,
so $V_a(0)\ne0$, and makes all mapped trivial zeros positive.  It is a
convenient range for the product formulas, not a restriction on the
known boundary-integral family.

The inverse map and trivial-zero images are
\begin{equation}\label{eq:a-inverse}
 w_{\rho,a}=\phi_a^{-1}(\rho)
 =\frac{\rho+a-1/2}{\rho-a-1/2},\qquad
 z_{n,a}=\frac{4n+1-2a}{4n+1+2a}.
\end{equation}
As before, set
\begin{equation}\label{eq:a-defs}
 \begin{gathered}
 N_a(r)=\card\{n\ge1:z_{n,a}<r\},\quad
 T_a(r)=\prod_{z_{n,a}<r}\frac r{z_{n,a}},\\
 A_a(r)=\log|V_a(0)T_a(r)|,\quad
 I_a(r)=\mean{\log|V_a(re^{i\theta})|},\quad
 D_a=I_a-A_a.
 \end{gathered}
\end{equation}
Subscript $a=1$ recovers the preceding notation.

\begin{proposition}[Gamma quotient]\label{prop:gamma}
Let $\alpha_\pm=(5\pm2a)/4$ and $N=N_a(r)$.  Then
\begin{equation}\label{eq:gamma-product}
 T_a(r)=r^N
 \frac{\Gamma(N+\alpha_+)}{\Gamma(N+\alpha_-)}
 \frac{\Gamma(\alpha_-)}{\Gamma(\alpha_+)}.
\end{equation}
Moreover,
\begin{equation}\label{eq:Na-asymptotic}
 N_a(r)=\frac{a}{1-r}+O_a(1).
\end{equation}
\end{proposition}

\begin{proof}
Write each factor as
\[
 \frac1{z_{n,a}}
 =\frac{n+(1+2a)/4}{n+(1-2a)/4}
\]
and apply $\prod_{j=0}^{N-1}(x+j)=\Gamma(x+N)/\Gamma(x)$.
The condition $z_{n,a}<r$ is equivalent to
\[
 n<\frac{a(1+r)}{2(1-r)}-\frac14
   =\frac{a}{1-r}-\frac a2-\frac14,
\]
which proves \eqref{eq:Na-asymptotic}.
\end{proof}

\begin{remark}[Two telescoping cases]\label{rem:two-cases}
Since $\alpha_+-\alpha_-=a$, integral values of $a$ make the gamma
ratio a finite polynomial in $N$.  In the interval \eqref{eq:a-range}
these values are $a=1$ and $a=2$:
\begin{equation}\label{eq:two-telescopes}
 T_1(r)=r^{N_1(r)}\frac{4N_1(r)+3}{3},\qquad
 T_2(r)=r^{N_2(r)}\frac{(4N_2(r)+1)(4N_2(r)+5)}{5}.
\end{equation}
Thus $a=1$ is the simplest telescoping case, not the only one.
\end{remark}

\begin{theorem}[Parameterized finite part]\label{thm:parameter}
For $a$ in \eqref{eq:a-range}, define
\begin{equation}\label{eq:delta-a}
 \Delta_a=\sum_{\Rea\rho<1/2}m(\rho)
 \log\left|\frac{\rho-a-1/2}{\rho+a-1/2}\right|.
\end{equation}
This series converges, and
\begin{equation}\label{eq:Da-sum}
 D_a(r)=\sum_{\substack{\Rea\rho<1/2\\|w_{\rho,a}|<r}}
 m(\rho)\log\frac r{|w_{\rho,a}|}
 \ \uparrow\ \Delta_a.
\end{equation}
Put
\begin{equation}\label{eq:Ca}
 C_a=\log\bigl((a-1/2)\zeta(a+1/2)\bigr)
       +a\log\frac a\pi-a.
\end{equation}
Then, unconditionally,
\begin{equation}\label{eq:a-asymptotic}
 A_a(r)=-a\log(1-r)+C_a+o(1),\qquad
 I_a(r)=-a\log(1-r)+C_a+\Delta_a+o(1).
\end{equation}
For each fixed such $a$, all the following are equivalent:
\[
 \RH;\qquad D_a(r)=0\ (0<r<1);\qquad \Delta_a=0;\qquad
 \lim_{r\uparrow1}\bigl(I_a(r)+a\log(1-r)\bigr)=C_a.
\]
\end{theorem}

\begin{proof}
Jensen's formula gives \eqref{eq:Da-sum}.  For
$\rho=1/2-\delta+i\gamma$ with $0<\delta<1/2$, its limiting
contribution satisfies
\begin{equation}\label{eq:qa}
 \begin{split}
 q_a(\rho)
 &=\frac12\log\frac{\gamma^2+(a+\delta)^2}
                         {\gamma^2+(a-\delta)^2}\\
 &\le\frac{2a\delta}{\gamma^2+(a-\delta)^2}
 \le\frac{a}{\gamma^2+(a-1/2)^2}.
 \end{split}
\end{equation}
The convergence argument of Lemma~\ref{lem:delta} applies.

For fixed $a$, the gamma-ratio asymptotic
$\Gamma(N+\alpha_+)/\Gamma(N+\alpha_-)=N^a(1+O_a(N^{-1}))$
\cite[\S5.11(iii)]{DLMF}, together with
\eqref{eq:Na-asymptotic}, gives
\begin{equation}\label{eq:logTa}
 \log T_a(r)=-a\log(1-r)+a\log a-a
       +\log\frac{\Gamma(\alpha_-)}{\Gamma(\alpha_+)}+o(1).
\end{equation}
To simplify the constant, write the functional equation for $u$ as
\begin{equation}\label{eq:Lambda}
 u(s)=\Lambda(s)u(1-s),\qquad
 \Lambda(s)=\pi^{s-1/2}
       \frac{\Gamma((3-s)/2)}{\Gamma(1+s/2)}.
\end{equation}
This is an immediate rearrangement of the completed zeta functional
equation; the same normalization is used in \cite[p.~836]{EG2015}.
At $s=1/2-a$,
\begin{equation}\label{eq:center-functional}
 u(1/2-a)=\pi^{-a}
   \frac{\Gamma(\alpha_+)}{\Gamma(\alpha_-)}
   (a-1/2)\zeta(a+1/2)>0.
\end{equation}
Adding its logarithm to \eqref{eq:logTa} cancels the gamma constants
and gives \eqref{eq:Ca}.  The asymptotic for $I_a$ follows from
\eqref{eq:Da-sum}.  Positivity of the zero contributions and symmetry
about the critical line prove the equivalences.
\end{proof}

\begin{remark}[Geometry of the count]
Writing $s=\sigma+it$, the condition $|\phi_a^{-1}(s)|<r$ is
\begin{equation}\label{eq:geometry}
 \left(\sigma-\frac12+
      a\frac{1+r^2}{1-r^2}\right)^2+t^2
 <\left(\frac{2ar}{1-r^2}\right)^2.
\end{equation}
Thus the finite-radius defect counts nontrivial zeros in an increasing
family of disks exhausting the left half-plane.  The analogue of
\eqref{eq:integrated-count} holds with $w_{\rho,a}$ in place of $w_\rho$.
\end{remark}

\section{The boundary integral and the earlier Jensen--Hardy proof}
\label{sec:hardy}

\subsection{Removing the trivial zeros analytically}

Set
\begin{equation}\label{eq:QWG}
 Q_a=\Lambda\circ\phi_a,\qquad
 W_a=u\circ\psi_a,\qquad
 G_a(z)=\frac{W_a(z)}{\psi_a(z)}
       =\frac{\psi_a(z)-1}{\psi_a(z)}\zeta(\psi_a(z)).
\end{equation}
The function $\Lambda$ in \eqref{eq:Lambda} is holomorphic on
$\Rea s<1/2$ and has precisely the simple zeros $-2,-4,\ldots$
there: the numerator gamma factor has no poles or zeros in that
half-plane, and $1/\Gamma(1+s/2)$ supplies those zeros.
Consequently,
\begin{equation}\label{eq:factorization}
 V_a=Q_aW_a
\end{equation}
is an analytic separation of the complete trivial-zero contribution.
Jensen's formula gives
\begin{equation}\label{eq:Qmean}
 \mean{\log|Q_a(re^{i\theta})|}
  =\log Q_a(0)+\log T_a(r),
 \qquad
 Q_a(0)=\pi^{-a}\frac{\Gamma(\alpha_+)}{\Gamma(\alpha_-)}.
\end{equation}
In particular,
\begin{equation}\label{eq:Q-asymptotic}
 \mean{\log|Q_a(re^{i\theta})|}
 =-a\log(1-r)+a\log\frac a\pi-a+o(1).
\end{equation}
The factor $\psi_a$ is holomorphic and zero-free on $\D$, so the
circular mean of $\log|\psi_a|$ is $\log(a+1/2)$.  Subtracting
\eqref{eq:Qmean} from the mean of \eqref{eq:factorization} therefore
yields the exact finite-radius identity
\begin{equation}\label{eq:G-jensen}
 D_a(r)=\mean{\log|G_a(re^{i\theta})|}-\log G_a(0),
 \qquad
 G_a(0)=\frac{a-1/2}{a+1/2}\zeta(a+1/2)>0.
\end{equation}

\subsection{Hardy-space justification of the boundary passage}

For clarity, $H^2(\D)$ consists of the holomorphic functions $f$ satisfying
\[
 \sup_{0<r<1}\mean{|f(re^{i\theta})|^2}<\infty.
\]
Write $f^*$ for the almost-everywhere radial boundary function.
We use the canonical factorization $f=BSO$, with $B$ a Blaschke
product, $S$ singular inner, and $O$ outer; see \cite{Duren1970}.
For $f(0)\ne0$,
\begin{equation}\label{eq:hardy-gap}
 \mean{\log|f^*(e^{i\theta})|}-\log|f(0)|
 =-\log|B(0)|-\log|S(0)|.
\end{equation}
The boundary integral is finite.  Equality of that integral with
$\log|f(0)|$ expresses outerness, not innerness.  A singular inner
factor has to be excluded before a boundary equality can be inferred
from the absence of disk zeros.

\begin{lemma}\label{lem:hardy}
For $a$ in \eqref{eq:a-range}, $G_a\in H^2(\D)$ and its singular
inner factor is constant.
\end{lemma}

\begin{proof}
We give the standard Hardy-space argument in the present coordinates;
compare \cite[Section~2]{Burnol2001}.  Put
\[
 h(s)=\frac{(s-1)\zeta(s)}{s^2}.
\]
Writing $\{x\}=x-\lfloor x\rfloor$, the Mellin identity
\[
 F(s):=\int_0^1\{1/t\}t^{s-1}\dd t
       =\frac1{s-1}-\frac{\zeta(s)}s
\]
holds initially for $\Rea s>1$ and extends to $\Rea s>1/2$.
Since $\{1/t\}\in L^2(0,1)$, the half-plane Paley--Wiener theorem
places $F$ in $H^2(\Rea s>1/2)$.  Also
\[
 h(s)=\frac1s-\frac{s-1}{s}F(s).
\]
The multiplier $(s-1)/s$ is bounded by $1$ in this half-plane and
$1/s$ belongs to its $H^2$ space.  Hence $h$ belongs to that space.
The conformal Hardy-space transfer gives
$h(\psi_a(z))/(1-z)\in H^2(\D)$, up to an inessential nonzero
normalizing constant.  Since
\[
 G_a(z)=\bigl((a+1/2)+(a-1/2)z\bigr)
              \frac{h(\psi_a(z))}{1-z},
\]
we obtain $G_a\in H^2(\D)$.

The function $G_a$ extends holomorphically across every point of the
unit circle except $1$.  The standard support property for singular
inner factors of a function continuing across an arc therefore leaves
at most a point mass at $1$; isolated boundary zeros do not create a
singular inner factor.  Such a mass $\lambda>0$ would contribute
\[
 S(z)=\exp\!\left(-\lambda\frac{1+z}{1-z}\right).
\]
If $G_a=BSO$, the point-evaluation bound for $O\in H^2(\D)$ would give
\[
 |G_a(r)|\le
 \exp\!\left(-\lambda\frac{1+r}{1-r}\right)
 \frac{\|O\|_{H^2}}{\sqrt{1-r^2}}\longrightarrow0.
\]
But $\psi_a(r)\to+\infty$ and $G_a(r)\to1$.  Thus $\lambda=0$.
\end{proof}

\begin{theorem}[Known boundary family in the present notation]
\label{thm:boundary}
Define
\begin{equation}\label{eq:Ma}
 M_a=\frac a\pi\int_{\R}
       \frac{\log|\zetahalf|}{a^2+t^2}\dd t,
 \qquad
 c_a=\log\left(\frac{a-1/2}{a+1/2}\zeta(a+1/2)\right).
\end{equation}
The integral converges absolutely and
\begin{equation}\label{eq:boundary-family}
 M_a=c_a+\Delta_a.
\end{equation}
In particular,
\begin{equation}\label{eq:boundary-rh}
 \RH\quad\Longleftrightarrow\quad M_a=c_a
 \quad\Longleftrightarrow\quad G_a\text{ is outer}.
\end{equation}
The circular finite part is related to this boundary integral by
\begin{equation}\label{eq:finite-boundary}
 \lim_{r\uparrow1}\bigl(I_a(r)+a\log(1-r)\bigr)
 =a\log\frac a\pi-a+\log(a+1/2)+M_a.
\end{equation}
\end{theorem}

\begin{proof}
Let $B_a$ be the Blaschke factor of $G_a$.  Its zeros are the disk
images, under $\psi_a^{-1}$, of the nontrivial zeros to the right of
the critical line.  Reflection across that line identifies their
moduli, with multiplicities, with those in \eqref{eq:a-inverse}
on the left.  Thus
\[
 -\log|B_a(0)|=\Delta_a.
\]
Equations~\eqref{eq:hardy-gap} and Lemma~\ref{lem:hardy} give
\[
 \mean{\log|G_a^*(e^{i\theta})|}=c_a+\Delta_a.
\]
For $0<\theta<2\pi$,
\[
 \psi_a(e^{i\theta})=\frac12+ia\cot(\theta/2),\qquad
 \left|\frac{\psi_a(e^{i\theta})-1}{\psi_a(e^{i\theta})}\right|=1.
\]
The substitution $t=a\cot(\theta/2)$ gives
$|\dd\theta|=2a\dd t/(a^2+t^2)$ and converts the boundary mean
into $M_a$.  Integrability of $\log|G_a^*|$ supplies absolute
convergence.  Positivity of \eqref{eq:delta-a} proves
\eqref{eq:boundary-rh}.  Finally,
$C_a-c_a=a\log(a/\pi)-a+\log(a+1/2)$, which yields
\eqref{eq:finite-boundary}.
\end{proof}

At $a=1$ these formulas become
\begin{equation}\label{eq:unit-boundary}
 \frac1\pi\int_{\R}\frac{\log|\zetahalf|}{1+t^2}\dd t
 =\log\frac{\zeta(3/2)}3+\Delta
\end{equation}
and
\begin{equation}\label{eq:unit-limit-boundary}
 \lim_{r\uparrow1}\bigl(I(r)+\log(1-r)\bigr)
 =\log\frac3{2\pi}-1+
 \frac1\pi\int_{\R}\frac{\log|\zetahalf|}{1+t^2}\dd t.
\end{equation}
Equation~\eqref{eq:unit-boundary} is explicitly recorded in
\cite[p.~840]{EG2015}.  The general identity
\eqref{eq:boundary-family} follows by evaluating the earlier
factorization of \cite[eq.~(2.2)]{Burnol2001} at $s=a+1/2$;
see also \cite[eqs.~(1.2), (1.4), and (1.6)]{Kunik2008}.

\begin{remark}[The circular mean is not a finite-radius Poisson integral]
For $r<1$,
\[
 \Rea\phi_a(re^{i\theta})
 =\frac12-a\frac{1-r^2}{|1-re^{i\theta}|^2}<\frac12.
\]
A change of the angular variable therefore cannot turn $I_a(r)$ into
an integral of $\log|u(1/2+it)|$ on the critical line.  The passage
above uses analytic factorization and the boundary values of $G_a$.
It does not interchange the limit and the original divergent integral.
\end{remark}

\subsection{Exact comparison with the earlier disk function}
\label{sec:comparison}

The disk function used in \cite{BSY1999} is
\[
 f_{\rm BSY}(w)=\frac{w}{1-w}\zeta\!\left(\frac1{1-w}\right),
\]
with the removable value $f_{\rm BSY}(0)=1$.  Put
\begin{equation}\label{eq:tau}
 \kappa_a=\frac{a-1/2}{a+1/2},\qquad
 \tau_a(z)=\frac{z+\kappa_a}{1+\kappa_a z}.
\end{equation}
A direct calculation gives
\begin{equation}\label{eq:exact-comparison}
 \psi_a(z)=\frac1{1-\tau_a(z)},\qquad
 W_a=f_{\rm BSY}\circ\tau_a.
\end{equation}
Consequently the defect remaining after the trivial-zero factor is
removed is the Jensen defect of the earlier function composed with a
disk automorphism.  For $a=1$, $\tau_1(z)=(1+3z)/(3+z)$.
This identifies the present work as a left-half-plane formulation of
the existing method.  Its explicit formulas are
\eqref{eq:telescoping}, \eqref{eq:gamma-product}, and the associated
finite parts, rather than a new boundary criterion.

\section{An absolutely convergent logarithmic integral of the completed function}
\label{sec:xi-bridge}

We next derive the real-axis identity needed for an information-theoretic
reformulation.  The starting point is the established boundary family in
Theorem~\ref{thm:boundary}; the completion and the choice of weight require
some care because a single Poisson-weighted logarithmic mean of $\Xi$
does not converge.

\subsection{Completion and the gamma contribution}

Use the standard normalization
\begin{equation}\label{eq:xi-definition}
 \xi(s)=\frac12s(s-1)\pi^{-s/2}\Gamma(s/2)\zeta(s),
 \qquad \Xi(t)=\xi(1/2+it).
\end{equation}
The completed function satisfies $\xi(s)=\xi(1-s)$; consequently $\Xi$
is even, entire, and real on $\R$ \cite[\S25.4]{DLMF}.  In particular,
\begin{equation}\label{eq:xi-imaginary}
 \Xi(ia)=\xi(1/2+a)>0\qquad(a>1/2).
\end{equation}
For real $t$, completion gives
\begin{equation}\label{eq:xi-modulus}
 |\Xi(t)|=\frac{t^2+1/4}{2\pi^{1/4}}
       |\Gamma(1/4+it/2)|\,|\zetahalf|.
\end{equation}
At a real zero the logarithm is interpreted by its locally integrable
singularity, as in Lemma~\ref{lem:jensen}.

\begin{lemma}[A regularized gamma mean]\label{lem:gamma-mean}
For $b,c>0$,
\begin{equation}\label{eq:gamma-mean}
 \frac b\pi\int_{\R}
 \frac{\log|\Gamma(c+ix)|+\frac\pi2|x|}{b^2+x^2}\dd x
 =\log\Gamma(c+b)+b(1-\log b).
\end{equation}
The integral converges absolutely.
\end{lemma}

\begin{proof}
On the right half-plane choose analytic logarithms that are real on
the positive real axis, and set
\[
 H_c(z)=\log\Gamma(c+z)-(z+c-1/2)\Log z+z.
\]
Stirling's expansion \cite[\S5.11]{DLMF} gives
$H_c(z)=\tfrac12\log(2\pi)+O_c(|z|^{-1})$ at infinity in that
half-plane.  Near zero $H_c(z)=O_c(1+|\log|z||)$.
The Poisson formula therefore applies to $\Rea H_c$; it can also
be obtained by Cauchy's formula on right half-annuli, whose inner
arc contributes $O(\varepsilon|\log\varepsilon|)$ and whose outer
arc tends to zero after multiplication by the Poisson-Cauchy kernel.
On the imaginary axis,
\[
 \Rea H_c(ix)=\log|\Gamma(c+ix)|+\frac\pi2|x|
                  -(c-1/2)\log|x|.
\]
Using
$\frac b\pi\int_{\R}\log|x|\dd x/(b^2+x^2)=\log b$
in the Poisson formula at $z=b$ yields \eqref{eq:gamma-mean}.
Stirling's formula shows that its numerator is $O_c(\log|x|)$ at
infinity; the behavior near zero is integrable.
\end{proof}

Define, for $a\in(1/2,5/2)$,
\begin{equation}\label{eq:regularized-Xi}
 \mathcal R_a=\frac a\pi\int_{\R}
 \frac{\log|\Xi(t)|+\frac\pi4|t|}{a^2+t^2}\dd t.
\end{equation}
By \eqref{eq:xi-modulus}, Lemma~\ref{lem:gamma-mean} with
$b=a/2$, $c=1/4$, and
\[
 \frac a\pi\int_{\R}\frac{\log(t^2+1/4)}{a^2+t^2}\dd t
       =2\log(a+1/2),
\]
we obtain
\begin{align}\label{eq:regularized-Xi-expansion}
 \mathcal R_a={}&M_a-\log2-\tfrac14\log\pi
       +2\log(a+1/2)+\log\Gamma(1/4+a/2)\notag\\
       &\hspace{35mm}+\frac a2\bigl(1-\log(a/2)\bigr).
\end{align}
Substituting $M_a=c_a+\Delta_a$ and collecting the completion
factors gives the unconditional identity
\begin{equation}\label{eq:regularized-Xi-identity}
 \mathcal R_a=\log\Xi(ia)
       +\frac a2\left(1+\log\frac{2\pi}{a}\right)+\Delta_a.
\end{equation}
Absolute convergence follows from Theorem~\ref{thm:boundary}, the
lemma, and the elementary logarithmic term.  Without the counterterm
$\pi|t|/4$, the corresponding Poisson integral diverges to $-\infty$.
Thus applying an unqualified harmonic mean-value formula directly
to $\log|\Xi|$ would lose a contribution from infinity.

\subsection{The quartic weight and a positive zero correction}

Put
\begin{equation}\label{eq:quartic-weight}
 w(t)=\frac{3}{\pi(1+t^2)(4+t^2)}
     =\frac1\pi\left(\frac1{1+t^2}-\frac1{4+t^2}\right),
 \qquad \int_{\R}w(t)\dd t=\frac12,
\end{equation}
and define
\begin{equation}\label{eq:K-Xi}
 \mathcal K=\int_{\R}w(t)\log|\Xi(t)|\dd t,
 \qquad \mathcal G=\Delta_1-\frac12\Delta_2.
\end{equation}
The quartic decay makes $\mathcal K$ absolutely convergent:
the gamma logarithm grows at most linearly, while the zeta logarithm
is integrable with the larger Poisson weight by
Theorem~\ref{thm:boundary}.

\begin{proposition}[The completed-function identity]\label{prop:Xi-gap}
Unconditionally,
\begin{equation}\label{eq:K-gap}
 \mathcal K=\log\Xi(i)-\frac12\log\Xi(2i)+\mathcal G,
 \qquad 0\le\mathcal G<\infty.
\end{equation}
Moreover,
\begin{equation}\label{eq:K-rh}
 \RH\quad\Longleftrightarrow\quad \mathcal G=0
 \quad\Longleftrightarrow\quad
 \mathcal K=\log\Xi(i)-\frac12\log\Xi(2i).
\end{equation}
\end{proposition}

\begin{proof}
Since $\int_{\R}w(t)|t|\dd t=2\log2/\pi$,
\[
 \mathcal R_1-\tfrac12\mathcal R_2
       =\mathcal K+\tfrac12\log2.
\]
The constants in \eqref{eq:regularized-Xi-identity} have the same
difference $\tfrac12\log2$, proving \eqref{eq:K-gap}.

It remains to check positivity of the difference of zero sums;
positivity of $\Delta_1$ and $\Delta_2$ separately would not suffice.
For a zero $\varrho=1/2-\delta+i\gamma$, $0<\delta<1/2$, write
\[
 q_a=\frac12\log\frac{\gamma^2+(a+\delta)^2}
                         {\gamma^2+(a-\delta)^2}
     =\operatorname{arctanh}x_a,
 \qquad x_a=\frac{2a\delta}{a^2+\delta^2+\gamma^2}.
\]
For $a>\delta$ we have $0<x_a<1$.  Direct differentiation gives
\[
 \frac{a q_a'}{x_a}
 =\frac{\delta^2+\gamma^2-a^2}
        {(a^2+\delta^2+\gamma^2)(1-x_a^2)}<1,
\]
where the strict inequality is equivalent to
$a^2+\gamma^2>\delta^2$.  Thus
$a q_a'<x_a<\operatorname{arctanh}x_a=q_a$, so $q_a/a$ is
strictly decreasing.  In particular $q_1-q_2/2>0$.
The series for $\Delta_1$ and $\Delta_2$ converge, and hence
\begin{equation}\label{eq:positive-quartic-gap}
 \mathcal G=
 \sum_{\Rea\varrho<1/2}m(\varrho)
             \left(q_1(\varrho)-\frac12q_2(\varrho)\right)
\end{equation}
is a convergent sum of strictly positive contributions.
It vanishes precisely when there are no left-hand nontrivial zeros.
Reflection in the critical line completes the proof.
\end{proof}

\section{A relative-entropy reformulation and lossless data processing}
\label{sec:kl}

In this section $\rho$ denotes a probability measure; zeros, when
needed, are denoted by $\varrho$.  The constant $D$ introduced below
is distinct from the finite-radius function $D(r)=I(r)-A(r)$.
The construction packages Proposition~\ref{prop:Xi-gap} as equality
in a bound for relative entropy.  The data-processing step is exact
without RH; it does not supply the zero-theoretic equality on its own.

The use of probability distributions associated with $\xi$ and $\Xi$ predates
the present construction; see in particular
\cite{BianePitmanYor2001,Nakamura2015,NakamuraSuzuki2023,KonstantopoulosPatieSarkar2024}.
Accordingly, the result below is presented as a relative-entropy reformulation
of the known analytic criterion, without a claim that probabilistic
formulations of RH are new in themselves.

\subsection{The Fourier kernel and the two probability spaces}

Riemann's Fourier representation, in the normalization recorded in
\cite[eqs.~(1.1)--(1.2)]{LagariasMontague2011}, is
\begin{equation}\label{eq:Phi-Fourier}
 \Xi(t)=2\int_0^\infty\Phi(u)\cos(tu)\dd u,
\end{equation}
where
\begin{equation}\label{eq:Phi-definition}
 \Phi(u)=\sum_{n=1}^\infty
  \left(4\pi^2n^4e^{9u/2}-6\pi n^2e^{5u/2}\right)
             e^{-\pi n^2e^{2u}},\qquad u\ge0.
\end{equation}
Each summand is positive, since it factors as
$2\pi n^2e^{5u/2}(2\pi n^2e^{2u}-3)e^{-\pi n^2e^{2u}}$.
The kernel has faster-than-exponential decay.  It follows that
\begin{equation}\label{eq:p-f}
 p(u)=\frac{\Phi(|u|)}{\Xi(0)},\qquad
 f(t)=\int_{\R}e^{itu}p(u)\dd u=\frac{\Xi(t)}{\Xi(0)}
\end{equation}
defines a symmetric probability density $p$ and its real-valued
characteristic function.  In particular $\Xi(0)>0$ and $|f(t)|\le1$.

Define a second probability measure on the frequency axis by
\begin{equation}\label{eq:rho-probability}
 \rho(\dd t)=2w(t)\dd t
       =\frac{6\dd t}{\pi(1+t^2)(4+t^2)}.
\end{equation}
For $T\sim\rho$, set
\begin{equation}\label{eq:Z-mu}
 Z(t)=f(t)^2=\left|\frac{\Xi(t)}{\Xi(0)}\right|^2,
 \qquad \mu=Z_\#\rho,
 \qquad M=\int_{[0,1]}z\,\mu(\dd z).
\end{equation}
Here $Z_\#\rho$ is the pushforward: $\mu(B)=\rho(Z^{-1}(B))$
for Borel $B\subset[0,1]$.  Thus $\mu$ is the law of $Z(T)$,
not the law of $T$.  For any integrable test function $H$,
\begin{equation}\label{eq:pushforward-expectation}
 \int_{[0,1]}H(z)\,\mu(\dd z)
       =\int_{\R}H(Z(t))\,\rho(\dd t).
\end{equation}
The strict inequality $|f(t)|<1$ for $t\ne0$ follows because $p$
is a positive density and $e^{itX}$ cannot be almost surely constant
when $t\ne0$.  The real zeros of the nonzero entire function $\Xi$
are discrete.  Hence $\mu(\{0\})=\mu(\{1\})=0$ and $0<M<1$.
By \eqref{eq:K-Xi} and \eqref{eq:pushforward-expectation},
\begin{equation}\label{eq:logZ-K}
 \int_{[0,1]}\log z\,\mu(\dd z)
    =4\mathcal K-2\log\Xi(0),
 \qquad \int_{[0,1]}|\log z|\,\mu(\dd z)<\infty.
\end{equation}
Values assigned to $\log Z$ on the zero set do not affect any integral.

The size-biased law on $[0,1]$ and its unreduced version on $\R$ are
\begin{equation}\label{eq:nu-rhoZ}
 \nu(\dd z)=\frac zM\,\mu(\dd z),
 \qquad \rho_Z(\dd t)=\frac{Z(t)}M\,\rho(\dd t).
\end{equation}
Both are probability measures, $Z_\#\rho_Z=\nu$, and the respective
pairs are mutually absolutely continuous.  For instance, if
$B\mid Z=z$ has Bernoulli success probability $z$, Bayes' formula
gives $\mathcal L(Z\mid B=1)=\nu$.  This interpretation does not
identify $D_{\rm KL}(\mu\|\nu)$ with ordinary mutual information.

\subsection{Equality in the data-processing inequality}

For probability measures $P,Q$, define
\[
 D_{\rm KL}(P\|Q)=\int\log\frac{\dd P}{\dd Q}\dd P
\]
when $P\ll Q$, and set it to $+\infty$ otherwise.  All logarithms
are natural.  The data-processing inequality (DPI) states that
$D_{\rm KL}(PK\|QK)\le D_{\rm KL}(P\|Q)$ for a common Markov
kernel $K$; deterministic pushforwards are a special case.
Equality is associated with sufficiency of the processed statistic;
see \cite[Chapters~2--3]{PolyanskiyWu2016}.

\begin{proposition}[Lossless reduction to $Z$]\label{prop:DPI}
For the measures in \eqref{eq:Z-mu} and \eqref{eq:nu-rhoZ},
\begin{equation}\label{eq:DPI-equality}
 D_{\rm KL}(\rho\|\rho_Z)
 =D_{\rm KL}(\mu\|\nu)
 =\log M-\int_{[0,1]}\log z\,\mu(\dd z)<\infty.
\end{equation}
The map $t\mapsto Z(t)$ is a sufficient statistic for the pair
$(\rho,\rho_Z)$, regardless of RH.
\end{proposition}

\begin{proof}
Apply DPI to the deterministic map $Z$:
\begin{equation}\label{eq:DPI-forward}
 D_{\rm KL}(\mu\|\nu)\le D_{\rm KL}(\rho\|\rho_Z).
\end{equation}
To prove the reverse inequality by DPI as well, let
$\mathsf R(z,\dd t)$ be a regular conditional distribution of $T$
given $Z(T)=z$ under $\rho$.  Such a kernel exists since these are
standard Borel spaces.  It reconstructs $\rho$ from $\mu$:
$\mu\mathsf R=\rho$.  The likelihood ratio
\[
 \frac{\dd\rho_Z}{\dd\rho}(t)=\frac{Z(t)}M
\]
is constant on every fiber of $Z$.  Consequently the same
conditional kernel reconstructs $\rho_Z$ from $\nu$: for Borel
$A\subset\R$,
\[
 (\nu\mathsf R)(A)
 =\frac1M\int z\,\mathsf R(z,A)\,\mu(\dd z)
 =\frac1M\int_A Z(t)\,\rho(\dd t)=\rho_Z(A).
\]
Applying DPI to $\mathsf R$ gives
\begin{equation}\label{eq:DPI-backward}
 D_{\rm KL}(\rho\|\rho_Z)
 =D_{\rm KL}(\mu\mathsf R\|\nu\mathsf R)
 \le D_{\rm KL}(\mu\|\nu).
\end{equation}
Together, \eqref{eq:DPI-forward} and \eqref{eq:DPI-backward}
prove equality.  Finally, $\dd\mu/\dd\nu=M/z$ almost everywhere,
so direct integration gives \eqref{eq:DPI-equality}; finiteness
follows from \eqref{eq:logZ-K}.
\end{proof}

\subsection{RH as saturation of a relative-entropy bound}

\begin{theorem}[Relative-entropy saturation form]\label{thm:KL}
Define
\begin{equation}\label{eq:KL-target-D}
 D=\log\left(M\frac{\Xi(0)^2\Xi(2i)^2}{\Xi(i)^4}\right).
\end{equation}
Then, without assuming RH,
\begin{equation}\label{eq:KL-gap}
 D-D_{\rm KL}(\mu\|\nu)
 =D-D_{\rm KL}(\rho\|\rho_Z)
 =4\mathcal G
 =4\left(\Delta_1-\frac12\Delta_2\right)\ge0.
\end{equation}
In particular,
\begin{equation}\label{eq:KL-bounds}
 0\le D_{\rm KL}(\mu\|\nu)\le D,
\end{equation}
and
\begin{equation}\label{eq:KL-RH}
 \boxed{\RH\quad\Longleftrightarrow\quad
        D_{\rm KL}(\mu\|\nu)=D.}
\end{equation}
Equivalently, the same equality holds for the unreduced pair
$(\rho,\rho_Z)$.
\end{theorem}

\begin{proof}
Combining Proposition~\ref{prop:DPI} with \eqref{eq:logZ-K} gives
\[
 D_{\rm KL}(\mu\|\nu)=\log M+2\log\Xi(0)-4\mathcal K.
\]
Substitute \eqref{eq:K-gap} to obtain \eqref{eq:KL-gap}.
Nonnegativity of relative entropy and
Proposition~\ref{prop:Xi-gap} give \eqref{eq:KL-bounds}.
By \eqref{eq:positive-quartic-gap}, $\mathcal G=0$ precisely under
RH, proving the equivalence.
\end{proof}

\begin{remark}[What DPI does and does not prove]\label{rem:DPI-scope}
The equality in \eqref{eq:DPI-equality} is unconditional: no
information about the likelihood ratio is lost by passing to $Z$.
The RH content lies in the additional assertion that this common
divergence equals the particular constant $D$.  That assertion
comes from the zero correction in Proposition~\ref{prop:Xi-gap},
not from equality in DPI.  Also, \eqref{eq:KL-bounds} is a bound
for these particular Riemann-derived measures; $D$ is not asserted
to be an upper bound over all probability pairs.  The theorem is
an information-theoretic reformulation of the established
logarithmic-integral family, not a proof of RH or an assertion of
bibliographic priority.
\end{remark}

\subsection{Variational form and the Fourier-kernel constant}
\label{sec:DV}

Set
\begin{equation}\label{eq:E-target}
 E=D-\log M=2\log\frac{\Xi(0)\Xi(2i)}{\Xi(i)^2}.
\end{equation}
Let $\mathcal B_b([0,1])$ denote the bounded real Borel functions.
For $h\in\mathcal B_b([0,1])$ put
\begin{equation}\label{eq:DV-functional}
 \mathcal J_\mu(h)=\int h(z)\,\mu(\dd z)
             -\log\int ze^{h(z)}\,\mu(\dd z).
\end{equation}
The Donsker--Varadhan variational formula
\cite[Theorem~3.5]{PolyanskiyWu2016}, applied to
$\nu(\dd z)=z\mu(\dd z)/M$, yields
\begin{equation}\label{eq:DV-sup}
 \sup_{h\in\mathcal B_b([0,1])}\mathcal J_\mu(h)
 =D_{\rm KL}(\mu\|\nu)-\log M
 =-\int\log z\,\mu(\dd z)=E-4\mathcal G.
\end{equation}
For completeness this instance can be checked directly.  Jensen's
inequality for $\log$ gives
$\mathcal J_\mu(h)\le-\int\log z\,\mu(\dd z)$.
For $L>0$ the bounded functions
$h_L(z)=\min\{-\log z,L\}$, with $h_L(0)=L$, satisfy
\[
 \int h_L\dd\mu\longrightarrow-\int\log z\dd\mu,
 \qquad
 \int ze^{h_L(z)}\,\mu(\dd z)\longrightarrow1.
\]
Monotone and dominated convergence prove \eqref{eq:DV-sup}.
In a larger admissible class the optimizer is
$h(z)=-\log z+$constant; it is not a bounded optimizer.
Consequently,
\begin{equation}\label{eq:DV-RH}
 \RH\quad\Longleftrightarrow\quad
 \sup_{h\in\mathcal B_b([0,1])}\mathcal J_\mu(h)=E.
\end{equation}
The equality of this supremum with $-\int\log z\dd\mu$ is true
without RH; only its equality with $E$ is the criterion.

Both sides can be expressed through the same positive Fourier
kernel.  Write
\[
 H_j=\int_0^\infty\Phi(u)\cosh(ju)\dd u\quad(j=0,1,2),
 \qquad
 R_\Phi(t)=\frac{\int_0^\infty\Phi(u)\cos(tu)\dd u}{H_0}.
\]
Then $Z(t)=R_\Phi(t)^2$, $\Xi(ji)=2H_j$, and
\begin{equation}\label{eq:E-Phi}
 E=2\log\frac{H_0H_2}{H_1^2}
 =2\log\frac{
  \left(\int_0^\infty\Phi(u)\dd u\right)
  \left(\int_0^\infty\Phi(u)\cosh(2u)\dd u\right)}{
  \left(\int_0^\infty\Phi(u)\cosh u\dd u\right)^2}.
\end{equation}
In particular, \eqref{eq:DV-RH} becomes
\begin{equation}\label{eq:DV-Phi}
 \RH\quad\Longleftrightarrow\quad
 \sup_{h\in\mathcal B_b([0,1])}
 \left\{\int h(R_\Phi(t)^2)\,\rho(\dd t)
 -\log\int R_\Phi(t)^2e^{h(R_\Phi(t)^2)}\,\rho(\dd t)\right\}
 =2\log\frac{H_0H_2}{H_1^2}.
\end{equation}
If $X$ has density $p$ from \eqref{eq:p-f}, symmetry and the rapid
decay of $p$ give $H_j/H_0=\mathbb E_p[e^{jX}]$.
Thus $H_0H_2/H_1^2$ is the normalized second moment
$\mathbb E_p[e^{2X}]/\mathbb E_p[e^X]^2$.
The laws of $X$ and $Z(T)$ are different: the latter uses the
characteristic function of the former at the random frequency $T$.

\subsection{Moment formulation and the sequences \texorpdfstring{$U_N,V_N$}{U(N) and V(N)}}
\label{sec:KL-moments}

For $k\ge0$ and $n\ge1$ define
\begin{equation}\label{eq:KL-moments}
 m_k=\int_{\R}w(t)Z(t)^k\dd t=\frac12\int z^k\,\mu(\dd z),
 \qquad
 A_n=\frac12\int(1-z)^n\,\mu(\dd z).
\end{equation}
In particular, $m_0=1/2$, $M=2m_1$, and
\[
 A_n=\sum_{k=0}^n(-1)^k\binom nk m_k\ge0.
\]
Here the sequence $A_n$ is distinct from the finite-radius function
$A(r)$.  The expansion $-\log z=\sum_{n\ge1}(1-z)^n/n$ and
monotone convergence give
\begin{equation}\label{eq:KL-series}
 D_{\rm KL}(\mu\|\nu)
 =\log(2m_1)+2\sum_{n=1}^\infty\frac{A_n}{n},
 \qquad
 \sum_{n=1}^\infty\frac{A_n}{n}<\infty.
\end{equation}
Equivalently,
\begin{equation}\label{eq:series-RH}
 \RH\quad\Longleftrightarrow\quad
 \sum_{n=1}^\infty\frac{A_n}{n}
       =\log\frac{\Xi(0)\Xi(2i)}{\Xi(i)^2}.
\end{equation}
The left side is at most the right side unconditionally, with gap
$2\mathcal G$.

To distinguish the two partial-sum normalizations, set
\begin{equation}\label{eq:UV-relation}
 V_N=-\frac12\sum_{n=1}^N\frac{A_n}{n},
 \qquad U_N=\frac12\log\Xi(0)+V_N.
\end{equation}
Then
\[
 U_N\downarrow\mathcal K,
 \qquad
 V_N\downarrow\mathcal K-\frac12\log\Xi(0).
\]
The symbol $V_N$ denotes this sequence and should not be confused
with the analytic function $V(z)$ in Section~\ref{sec:main}.

The moments also retain a direct convolution interpretation in
terms of $p$.  With
\[
 g(u)=e^{-|u|}-\frac12e^{-2|u|},
 \qquad
 \widehat g(t)=\int_{\R}g(u)e^{-itu}\dd u
             =\frac{6}{(1+t^2)(4+t^2)},
\]
Fourier inversion and \eqref{eq:p-f} give, for $k\ge1$,
\begin{equation}\label{eq:mk-convolution}
 m_k=\int_{\R}g(u)\,p^{*2k}(u)\dd u.
\end{equation}
For $k=0$ use the point mass at zero instead of a convolution
density.  These identities give concrete representations of the
moments; they do not by themselves prove the equality in
\eqref{eq:series-RH}.

\section{Concluding remarks}

The finite-radius construction makes one elementary feature unusually
explicit: in the normalization centered at $s=-1/2$, the Jensen product of
the mapped trivial zeros telescopes.  The parameterized version replaces this
by a gamma quotient.  In either case the divergent part of the circular mean
is completely explicit, while the remaining defect is a finite positive zero
sum.  This is a left-half-plane presentation of the established
Jensen--Hardy mechanism rather than a new boundary criterion.

The relative-entropy formulation packages the same obstruction in a different
language.  For the Riemann-derived pair $(\mu,\nu)$, data processing is
lossless because $Z$ retains the likelihood ratio, and
\[
 D-D_{\rm KL}(\mu\|\nu)=4\mathcal G\ge0.
\]
Thus RH is equivalent to saturation of this particular upper bound.  The
Donsker--Varadhan and moment representations clarify the probabilistic
structure but, without an independent lower bound reaching $D$, do not
constitute additional progress toward proving RH.

The potentially useful content of these reformulations is therefore
structural: the finite-radius formulas isolate the trivial-zero contribution,
and the completed-function identity places the remaining equality problem in
a relative-entropy setting connected with Riemann's positive Fourier kernel.
Whether that viewpoint yields estimates not already implicit in the
zero-theoretic formulation is left open.

\section*{Disclosure of AI assistance}
OpenAI ChatGPT was used as an assistance tool for drafting, algebraic checks, and locating related literature.
The author is responsible for the statements, proofs, citations, and conclusions of the manuscript.

\clearpage
\bibliographystyle{plain}
\bibliography{references}

\end{document}